\documentclass[11pt,a4paper]{article}
\usepackage{amsmath}
\usepackage{amsthm}
\usepackage{makeidx}
\usepackage{latexsym}
\usepackage{amssymb}

\usepackage[caption=false,font=footnotesize]{subfig}
\usepackage{floatflt}
\usepackage{graphicx}
\usepackage{bm}

\usepackage{epsfig}
\usepackage{epsf}

\usepackage{geometry}
\usepackage[latin1]{inputenc}
\usepackage{hyperref}
\usepackage{float}
\newtheorem{theorem}{Theorem}
\newtheorem{lemma}{Lemma}
\newtheorem{proposition}{Proposition}

\newtheorem{remark}{Remark}
\newtheorem{definition}{Definition}
\theoremstyle{remark}
\newtheorem{example}{\textbf{Example}}
\usepackage[latin1]{inputenc}
\title{Linear divisibility sequences and Salem numbers}
\author{Marco Abrate, Stefano Barbero, Umberto Cerruti, Nadir Murru\\ \\
Department of Mathematics, University of Turin\\
Via Carlo Alberto 10, 10122, Turin, ITALY\\ \\
marco.abrate@unito.it, stefano.barbero@unito.it, umberto.cerruti@unito.it,\\ nadir.murru@unito.it}
\date{}

\begin{document}

\maketitle

\begin{abstract}
We study linear divisibility sequences of order 4, providing a characterization
by means of their characteristic polynomials and finding their factorization
as a product of linear divisibility sequences of order 2. Moreover,
we show a new interesting connection between linear divisibility sequences
and Salem numbers. Specifically, we generate linear divisibility sequences
of order 4 by means of Salem numbers modulo 1. 
\end{abstract}

\section{Introduction}

A sequence $a=(a_{n})_{n=0}^{\infty}$ is a divisibility sequence
if $m|n$ implies $a_{m}|a_{n}$. Divisibility sequences that satisfy
a linear recurrence relation are particularly studied. A classic example
of linear divisibility sequence is the Fibonacci sequence. During
the years linear divisibility sequences of order 2 have been deeply
studied, see, e.g., \cite{Hor} and \cite{Nor}. Hall \cite{Hall}
studied divisibility sequences of order 3 and Bezivin et al. \cite{Bez}
have obtained more general results. Divisibility sequences are very
interesting for their beautiful properties. For example, many studies
can be found about their connection with elliptic curves \cite{Ward},
\cite{Ing}. Further results on divisibility sequences can be found,
e.g, in \cite{Cor} where Cornelissen and Reynolds investigate matrix
divisibility sequences, and in \cite{Yal} where Horak and Skula
characterize the second--order strong divisibility sequences.

Recently, linear divisibility sequences of order 4 have been deeply
examined. In particular, Williams and Guy \cite{Guy}, \cite{Guy2}
introduced and studied a class of linear divisibility sequences of
order 4 that extends the Lehmer--Lucas theory for divisibility sequences
of order 2. In section \ref{sec:standard}, we consider these sequences
proving that all (non degenerate) divisibility sequences of order
4 have characteristic polynomial equals to the characteristic polynomial
of sequences of Williams and Guy. Moreover, we provide all factorizations
of divisibility sequences of order 4 into the product of divisibility
sequences of order 2.

In section \ref{sec:salem}, we generate linear divisibility sequences
of order 4 by means of powers of Salem numbers. This result is particularly
intriguing, since connections between Salem numbers and divisibility
sequences have been never highlighted. Moreover, the construction
of divisibility sequences by means of powers of algebraic integers
is an interesting research field that have been recently developed \cite{Sil}.

\section{Standard linear divisibility sequences}

\label{sec:standard} \begin{definition} Given a ring $\mathcal{R}$,
a sequence $a=(a_{n})_{n=0}^{\infty}$ over $\mathcal{R}$ is a \emph{divisibility
sequence} if 
\[
m|n\Rightarrow a_{m}|a_{n}.
\]
Conventionally, we will consider $a_{0}=0$. \end{definition}

In the following, we will deal with \emph{linear divisibility sequences}
(LDSs), i.e., divisibility sequences that satisfy a linear recurrence.
Classic LDSs are the Lucas sequences, i.e., the linear recurrence
sequences with characteristic polynomial $x^{2}-hx+k$ and initial
conditions $0,1$.

In \cite{Guy} and \cite{Guy2}, the authors introduced and studied
some linear divisibility sequences of order 4. We recall these sequences
in the following definition.

\begin{definition} \label{standard} Let us consider linear recurrence
sequences of order 4 over $\mathbb{Z}$ with characteristic polynomial
\[
x^{4}-px^{3}+(q+2r)x^{2}-prx+r^{2}
\]
and initial conditions 
\[
0,1,p,p^{2}-q-3r.
\]
We say that these sequences are \emph{standard} LDSs of order 4 and
we call the previous polynomial as \emph{standard polynomial}. \end{definition}
In the next theorem, we prove that the product of two LDSs of order
2 is a standard LDS of order 4. First, we need the following lemma
proved in \cite{Cer}. 

\begin{lemma} \label{kron} Let $a=(a_{n})_{n=0}^{\infty}$
and $b=(b_{n})_{n=0}^{\infty}$ be linear recurrence sequences with
characteristic polynomials $f(x)$ and $g(x)$, respectively. The
sequence $ab=(a_{n}b_{n})_{n=0}^{\infty}$ is a linear recurrence
sequence that recurs with $f(x)\otimes g(x)$, the characteristic
polynomial of the matrix $F\otimes G$ (Kronecker product of matrices),
where $F$ and $G$ are the companion matrices of $f(x)$ and $g(x)$,
respectively. \end{lemma}

\begin{remark} \label{rem:kron}
The previous lemma can be also stated as follows. Let $a=(a_{n})_{n=0}^{\infty}$ and $b=(b_{n})_{n=0}^{\infty}$ be linear recurrence sequences whose characteristic polynomials have roots $\alpha_1,...,\alpha_s$ and $\beta_1,...,\beta_t$, respectively. Then, the sequence $c=(c_{n})_{n=0}^{\infty}=(a_{n}b_{n})_{n=0}^{\infty}$ is also a linear recurrence sequence whose characteristic polynomial has roots $\gamma_1,...,\gamma_{st}$, where
$$(\gamma_1,...,\gamma_{st}) = (\alpha_1,...,\alpha_s)\otimes(\beta_1,...,\beta_t),$$
\end{remark}

\begin{theorem} Let $a=(a_{n})_{n=0}^{\infty}$ and $b=(b_{n})_{n=0}^{\infty}$
be LDSs of order 2 with characteristic polynomials $x^{2}-h_{1}x+k_{1}$,
$x^{2}-h_{2}x+k_{2}$, respectively, and initial conditions $0,1$.
The sequence $ab=(a_{n}b_{n})_{n=0}^{\infty}$ is a standard LDS of
order 4 with initial conditions $0,1,h_{1}h_{2},(h_{1}^{2}-k_{1})(h_{2}^{2}-k_{2})$.
\end{theorem} \begin{proof} Since $a$ and $b$ are LDSs, it immediately
follows that $ab$ is a divisibility sequence and by Lemma \ref{kron},
we know that it is a linear recurrence sequence of order 4 whose characteristic
polynomial is 
\begin{equation*} 
x^{4}-h_{1}h_{2}x^{3}+(k_{1}h_{1}^{2}-k_{2}h_{1}^{2}+2k_{1}k_{2})x^{2}+h_{1}k_{1}h_{2}k_{2}x+k_{1}^{2}k_{2}^{2}.
\end{equation*}
By Definition \ref{standard}, $ab$ is a standard LDS for $p=h_{1}h_{2}$,
$q=h_{1}^{2}k_{2}+k_{1}(h_{2}^{2}-4k_{2})$, $r=k_{1}k_{2}$. The
initial conditions can be directly calculated. \end{proof}

Moreover, we prove that all the LDSs of order 4 have characteristic
polynomial equals to the characteristic polynomial of standard LDSs.

\begin{theorem} Let $a=(a_{n})_{n=0}^{\infty}$ be a non degenerate
LDS of order 4 with $a_{0}=0$ and $a_{1}=1$, then its characteristic
polynomial is 
\begin{equation} \label{eq:pol}
x^{4}-px^{3}+(q+2r)x^{2}-prx+r^{2}
\end{equation}
for some $p,q,r$. \end{theorem} \begin{proof} Let us suppose that
the characteristic polynomial of $a$ has distinct roots in order
to avoid degenerate sequences, i.e., ratio of roots are not roots
of unity. Let $\alpha$, $\beta$, $\gamma$, $\delta$ be these roots.

The sequence $a$ is a divisor of the sequence $b=(b_n)_{n=0}^\infty$, where
\[
b_n = \cfrac{\alpha^n-\beta^n}{\alpha-\beta}\cdot\cfrac{\alpha^n-\gamma^n}{\alpha-\gamma}\cdot\cfrac{\alpha^n-\delta^n}{\alpha-\delta}\cdot\cfrac{\beta^n-\gamma^n}{\beta-\gamma}\cdot\cfrac{\beta^n-\delta^n}{\beta-\delta}\cdot\cfrac{\gamma^n-\delta^n}{\gamma-\delta}.
\]
See \cite{Bar} and \cite{Bez}. In other words, there exist a sequence $c=(c_n)_{n=0}^\infty$ such that $b_n = a_nc_n$, for any index $n$.

By Lemma \ref{kron} and Remark \ref{rem:kron}, the sequence $p$ can be written as the product of six Lucas sequences with characteristic polynomials having roots $(\alpha, \beta)$, $(\alpha, \gamma)$, $(\alpha, \delta)$, $(\beta, \gamma)$, $(\beta, \delta)$, $(\gamma, \delta)$, respectively. Thus, the roots of the characteristic polynomial of $p$ are the entries of the following vector of length 64:
$$B = (\alpha, \beta)\otimes(\alpha,\gamma)\otimes(\alpha,\delta)\otimes(\beta,\gamma)\otimes(\beta,\delta)\otimes(\gamma,\delta),$$
where all the roots appear with the due multiplicity. We can write the vector $B$ as
$$B=(B_1, B_2, B_3, B_4),$$
where
\begin{itemize}
\item $B_1 = (\alpha^3, \alpha^2\delta)\otimes(\beta, \gamma)\otimes(\beta, \delta)\otimes(\gamma, \delta)$,
\item $B_2 = (\alpha^2, \alpha\delta)\otimes(\beta\gamma, \gamma^2)\otimes(\beta, \delta)\otimes(\gamma, \delta)$,
\item $B_3 = (\alpha^2\beta, \alpha\beta\delta)\otimes(\beta, \gamma)\otimes(\beta, \delta)\otimes(\gamma, \delta)$,
\item $B_4 = (\alpha\beta\gamma, \beta\gamma\delta)\otimes(\beta, \gamma)\otimes(\beta, \delta)\otimes(\gamma, \delta)$.
\end{itemize}

Moreover, $B = A \otimes C$, where $C$ is a vector whose components are the roots of the characteristic polynomial of $c$ and
$$A = (\omega_1, \omega_2, \omega_3, \omega_4),$$

with $(\omega_1, \omega_2, \omega_3, \omega_4)$ a certain permutation of $(\alpha, \beta, \gamma, \delta)$. Thus, we can write
$$B = (\omega_1 C, \omega_2 C, \omega_3 C, \omega_4 C),$$
i.e., $B_1$, $B_2$, $B_3$, and $B_4$ are multiple of $C$. Considering

$$C = (\alpha^2, \alpha\delta)\otimes(\beta, \gamma)\otimes(\beta, \delta)\otimes(\gamma, \delta),$$

we have $B_1 = \alpha C$, $B_2 = \gamma C$, $B_3 = \beta C$, $B_4 = \delta \cdot \cfrac{\beta\gamma}{\alpha\delta} C$. Thus, we have $\omega_1 = \alpha$, $\omega_2 = \gamma$, $\omega_3 = \beta$ and $\omega_4$ mus be equals to $\delta$, i.e., we must have $\alpha\delta = \beta\gamma$, but this is equivalent to say that the characteristic polynomial of $a$ must be of the form \eqref{eq:pol}.
 
\end{proof} 

Now, we see that any standard LDS can be factorized as
a product of two LDS of order 2 over $\mathbb{C}$ .

\begin{definition} Given the sequences $(u_{n})_{n=0}^{+\infty},(v_{n})_{n=0}^{+\infty},(s_{n})_{n=0}^{+\infty},(t_{n})_{n=0}^{+\infty}$
over a ring $\mathcal{R}$, we say that the product sequences $(u_{n}v_{n})_{n=0}^{+\infty}$
and $(s_{n}t_{n})_{n=0}^{+\infty}$ are \emph{equivalent} if 
\[
u_{n}=\lambda^{n-1}s_{n},\quad v_{n}=\lambda^{1-n}t_{n}
\]
where $\lambda\in\mathcal{R}$ is a unit. \end{definition}

\begin{theorem} Let $a=(a_{n})_{n=0}^{\infty}$ be a standard LDS
over $\mathbb{Z}$, then $a_{n}=b_{n}c_{n}$, for all $n\geq0$, where
$b=(b_{n})_{n=0}^{\infty}$ and $c=(c_{n})_{n=0}^{\infty}$ are LDSs
of order 2 over $\mathbb{C}$ with initial conditions $0,1$ and characteristic
polynomials 
\[
\begin{cases}
x^{2}-\cfrac{\sqrt{q+4r+2p\sqrt{r}}\pm\sqrt{q+4r-2p\sqrt{r}}}{2\sqrt{r}}x+1\\
x^{2}-\cfrac{\sqrt{q+4r+2p\sqrt{r}}\mp\sqrt{q+4r-2p\sqrt{r}}}{2}x+r
\end{cases}
\]
when $p\neq0$. Moreover when $p=0$ and $q+4r\neq0,$ $q\neq0$ (to
avoid degenerate cases) we have the two possible families of characteristic
polynomials for $b$ and $c$ given by

\[
\begin{cases}
x^{2}+1\\
x^{2}-\sqrt{q+4r}x+r
\end{cases},\quad\begin{cases}
x^{2}+1\\
x^{2}-\sqrt{q}x-r
\end{cases}
\]
These are all the families of not equivalent factorizations of $a$
over $\mathbb{C}$. \end{theorem} \begin{proof} We want to factorize
a standard polynomial into the Kronecker product of two polynomials
of degree 2, i.e., we want to find $h_{1},h_{2},k_{1},k_{2}$ such
that 
\[
(x^{2}-h_{1}x+k_{1})\otimes(x^{2}-h_{2}x+k_{2})=x^{4}-px^{3}+(q+2r)x^{2}-px+r^{2}.
\]
Let us observe that the characteristic polynomial of $a$ must have
distinct non zero roots in order to guarantee that $a$ is a LDS of
order 4 . Let $\gamma_{1},\gamma_{2}$ and $\sigma_{1},\sigma_{2}$
be the roots of $x^{2}-h_{1}x+k_{1}$ and $x^{2}-h_{2}x+k_{2}$, respectively.
We have 
\begin{equation}
\begin{cases}
(\gamma_{1}+\gamma_{2})(\sigma_{1}+\sigma_{2})=p\\
(\gamma_{1}^{2}+\gamma_{2}^{2})\sigma_{1}\sigma_{2}+\gamma_{1}\gamma_{2}(\sigma_{1}+\sigma_{2})^{2}=q+2r\\
\gamma_{1}\gamma_{2}\sigma_{1}\sigma_{2}(\gamma_{1}+\gamma_{2})(\sigma_{1}+\sigma_{2})=pr\\
(\gamma_{1}\gamma_{2}\sigma_{1}\sigma_{2})^{2}=r^{2}
\end{cases}\label{eq:sist}
\end{equation}
When $p\neq0$ these conditions are equivalent to the system 
\begin{equation}
\begin{cases}
k_{1}k_{2}=r\\
h_{1}h_{2}=p\\
h_{1}^{2}k_{2}+h_{2}^{2}k_{1}=q+4r
\end{cases}\label{sist1}
\end{equation}
which is a particular case of 
\[
\begin{cases}
k_{1}k_{2}=A\\
h_{1}h_{2}=B\\
h_{1}^{2}k_{2}+h_{2}^{2}k_{1}=C
\end{cases}
\]
where $A\not=0$ since we suppose that the standard polynomial has
not zero roots. Thus, we can obtain 
\[
A\left(\cfrac{h_{1}^{2}}{k_{1}}\right)^{2}-C\left(\cfrac{h_{1}^{2}}{k_{1}}\right)+B^{2}=0
\]
from which we have 
\[
h_{1}=\pm\sqrt{k_{1}}\cfrac{\sqrt{C+2B\sqrt{A}}\pm\sqrt{C-2B\sqrt{A}}}{2\sqrt{A}}
\]
and 
\[
h_{2}=\pm\cfrac{\sqrt{C+2B\sqrt{A}}\mp\sqrt{C-2B\sqrt{A}}}{2\sqrt{k_{1}}}.
\]
Thus solutions of system \ref{sist1} are 
\[
\begin{cases}
h_{1}=\pm\sqrt{k_{1}}\cfrac{\sqrt{q+4r+2p\sqrt{r}}\pm\sqrt{q+4r-2p\sqrt{r}}}{2\sqrt{r}}\\
h_{2}=\pm\cfrac{\sqrt{q+4r+2p\sqrt{r}}\mp\sqrt{q+4r-2p\sqrt{r}}}{2\sqrt{k_{1}}}\\
k_{2}=\cfrac{r}{k_{1}}
\end{cases}.
\]
Let us pose 

$\lambda=\pm\sqrt{k_{1}}$, $s=\cfrac{\sqrt{q+4r+2p\sqrt{r}}\pm\sqrt{q+4r-2p\sqrt{r}}}{2\sqrt{r}},$,
$\bar{s}=\cfrac{\sqrt{q+4r+2p\sqrt{r}}\mp\sqrt{q+4r-2p\sqrt{r}}}{2}$. 

\noindent Thus, considering solutions of system \ref{sist1}, we
have $x^{2}-h_{1}x+k_{1}=x^{2}-s\lambda x+\lambda^{2}$ and $x^{2}-h_{2}x+k_{2}=x^{2}-\frac{\bar{s}}{\lambda}x+\frac{r}{\lambda^{2}}$,
whose roots are 
\[
\gamma_{1,2}=\lambda\left(\cfrac{s\pm\sqrt{s^{2}-4}}{2}\right),\quad\sigma_{1,2}=\cfrac{1}{\lambda}\left(\cfrac{\bar{s}\pm\sqrt{\bar{s}^{2}-4}}{2}\right).
\]
In this case, we have $u_{n}=\lambda^{n-1}b_{n}$ and $v_{n}=\lambda^{1-n}c_{n}$,
where $b$ and $c$ are Lucas sequences with characteristic polynomials
$x^{2}-sx+1$ and $x^{2}-\bar{s}x+r$, respectively. When $p=0$ in
conditions (\ref{eq:sist}) we may suppose $\gamma_{1}+\gamma_{2}=h_{1}=0$
and find the two systems

\[
\begin{cases}
h_{1}=0\\
k_{1}k_{2}=r\\
h_{2}^{2}k_{1}=p+4r
\end{cases},\quad\begin{cases}
h_{1}=0\\
k_{1}k_{2}=-r\\
h_{2}^{2}k_{1}=p
\end{cases}
\]
with respective solutions

\[
\begin{cases}
h_{1}=0\\
h_{2}=\pm\sqrt{\frac{p+4r}{k_{1}}}\\
k_{2}=\frac{r}{k_{1}}
\end{cases},\quad\begin{cases}
h_{1}=0\\
h_{2}=\pm\sqrt{\frac{p}{k_{1}}}\\
k_{2}=-\frac{r}{k_{1}}
\end{cases}
\]
which give, with analogous considerations as in the case $p\neq0$,
with $\lambda=\pm\sqrt{k_{1}}$, the two families of characteristic
polynomials for $b$ and $c$ related to this case. \end{proof}\begin{remark}
It would be interesting to find when previous factorizations determine
sequences in $\mathbb{Z}$ or $\mathbb{Z}[i]$. \end{remark} In the
next section, we see a new connection between LDS of order 4 and Salem
numbers.

\section{Construction of linear divisibility sequences by means of Salem numbers
of order 4}

\label{sec:salem} The Salem numbers have been introduced in 1944
by Raphael Salem \cite{Salem} and they are closely related to the
Pisot numbers \cite{Pisot}. There are several results regarding Pisot
numbers and recurrence sequences \cite{Boyd0}, \cite{Boyd}, \cite{Boyd1}.
In the following, we relate Salem numbers and LDS.

There are many equivalent definitions of Salem numbers, here we report
the following one.

\begin{definition} A \emph{Salem number} is an algebraic integer
$\tau>1$ of degree $d\geq4$ such that all the conjugate elements
belong to the unitary circle, unless $\tau$ and $\tau^{-1}$. \end{definition}

In the following, we work on Salem numbers of degree 4, which can
be characterized as follows (see \cite{Bertin}, pag. 81). \begin{proposition}
The Salem numbers of degree 4 are all the real roots $\tau>1$, of
the following polynomials with integer coefficients 
\[
x^{4}+tx^{3}+cx^{2}-tx+1
\]
where 
\[
2(t-1)<c<-2(t+1).
\]
\end{proposition}

It is immediate to see that previous polynomials are standard polynomials
for $p=-t$, $q=-2+c$, $r=1$.

\begin{definition} We call \emph{Salem standard polynomials} the
polynomials 
\[
x^{4}-px^{3}+(q+2)x^{2}-px+1
\]
with 
\[
2(-p-1)<2+q<-2(-p+1).
\]
\end{definition}

The study of the distribution modulo 1 of the powers of a given real
number greater than 1 is a rich and classic research field (see, e.g,
\cite{Kok}). In the following, we use the same notation of \cite{Bertin}
(pag. 61).

\begin{definition} Given a real number $\alpha$, let $E(\alpha)$
be the nearest integer to $\alpha$, i.e., $\alpha=E(\alpha)+\epsilon(\alpha)$
where $\epsilon(\alpha)\in[-\frac{1}{2},\frac{1}{2}]$ is called \emph{$\alpha$
modulo 1}. \end{definition}

In the original work of Salem \cite{Salem}, he proved that if $\alpha$
is a Pisot number, then $\alpha^{n}$ modulo 1 tends to zero and if
$\alpha$ is a Salem number, then $\alpha^{n}$ modulo 1 is dense
in the unit interval. Further results on the distribution modulo 1
of the Salem numbers can be found, e.g., in \cite{Aki} and \cite{Zaimi2}.
Moreover, integer and fractional parts of Pisot and Salem numbers
have been studied, e.g., in \cite{Dub} and \cite{Zaimi}.

Let $\mathcal{R}\subseteq\mathbb{C}$ be a ring and $\alpha\in\mathcal{R}$
with $\alpha\not\in\mathcal{R}^{*}$, then the sequence $(\alpha^{n})_{n=0}^{\infty}$
is clearly a LDS. Given a couple of irrational numbers $\lambda$
and $\alpha$, it is interesting to study when the sequence $(E(\lambda\alpha^{n}))_{n=0}^{\infty}$
is a LDS.

\begin{example} If we consider $\frac{1}{\sqrt{5}}$ and the golden
mean $\phi$, it is well--known that 
\[
E\left(\frac{1}{\sqrt{5}}\phi^{n}\right)=F_{n},
\]
where $F_{n}$ is the $n$--th Fibonacci number, consequently we get
a LDS. \end{example}

Let $g(x)$ be a Salem standard polynomial, $g(x)$ has real roots
$\alpha>1$, $\alpha^{-1}$ and complex roots $\gamma$, $\gamma^{-1}$
with norm 1. Let $(u_{n})_{n=0}^{\infty}$ be a standard LDS with
characteristic polynomial $g(x)$. By the Binet formula, there exist
$\lambda,\lambda_{1},\lambda_{2},\lambda_{3}$ such that 
\[
u_{n}=\lambda\alpha^{n}+\lambda_{1}\alpha^{-n}+\lambda_{2}\gamma^{n}+\lambda_{3}\gamma^{-n}.
\]
Since 
\[
\lvert u_{n}-\lambda\alpha^{n}\rvert\geq\lvert\lambda_{1}\alpha^{-n}\rvert+\lvert\lambda_{2}\rvert+\lvert\lambda_{3}\rvert,
\]
for all $\epsilon>0$, with $n$ sufficiently large, we have 
\[
\lvert u_{n}-\lambda\alpha^{n}\rvert\geq\epsilon+\lvert\lambda_{2}\rvert+\lvert\lambda_{3}\rvert.
\]
Thus, if $\lvert\lambda_{2}\rvert+\lvert\lambda_{3}\rvert<\frac{1}{2}$,
there exists $n_{0}$ such that 
\[
u_{n}=E(\lambda\alpha^{n}),\quad\forall n>n_{0}
\]
and if $\lvert\lambda_{1}\alpha^{-1}\rvert+\lvert\lambda_{2}\rvert+\lvert\lambda_{3}\rvert<\frac{1}{2}$,
then 
\[
u_{n}=E(\lambda\alpha^{n}),\quad\forall n\geq1.
\]

An interesting case is given by the Salem standard polynomial 
\[
x^{4}-tx^{3}+tx^{2}-tx+1
\]
for $t\geq6$. In this case, we have the Salem numbers 
\[
\alpha=\cfrac{1}{4}\left(t+\sqrt{(t-4)t+8}+\sqrt{2}\sqrt{t(t+\sqrt{(t-4)t+8}-2)-4}\right)
\]
and 
\[
\lambda=\cfrac{1}{\sqrt{(t-4)t+8}}.
\]
Thus, we can determine infinitely many LDSs generated by powers of
a Salem number, specifically the sequences 
\[
(\theta_{n}(t))_{n=1}^{\infty}=E(\lambda\alpha^{n}),\quad\forall t\geq6\in\mathbb{Z}
\]
For example, when $t=6$ we have the LDS 
\[
1,6,29,144,725,3654,18409,...
\]
when $t=7$, we have 
\[
1,7,41,245,8897,53621,...
\]
These sequences appear to be new, since they are not listed in OEIS
\cite{oeis}. Moreover, as a consequence, we have the following property
on Salem numbers, i.e., 
\[
d|n\Rightarrow E(\lambda\alpha^{d})|E(\lambda\alpha^{n}).
\]

Finally, in the following proposition we characterize all the Salem
standard polynomials that produces LDSs of this kind.

\begin{proposition} With the above notation, if $\lvert\lambda_{1}\alpha^{-1}\rvert+\lvert\lambda_{2}\rvert+\lvert\lambda_{3}\rvert<\frac{1}{2}$,
then the integer coefficients $p,q$ of $g(x)$ must satisfy the following
inequalities 
\[
\begin{cases}
2\leq p\leq8,\quad-4-2p<q<\cfrac{p^{4}+8p^{3}-160p-400}{4p^{2}+32p+64}\\
p>8,\quad-4-2p<q<-4+2p
\end{cases}
\]
\end{proposition} \begin{proof} The real root $\alpha>1$ of $g(x)$
can be written as 
\[
\alpha=\cfrac{1}{4\left(p+\sqrt{p^{2}-4q}+\sqrt{(p+\sqrt{p^{2}-4q})^{2}-16}\right)}.
\]
Moreover, by the Binet formula 
\[
\lambda=\lambda_{1}=\cfrac{\alpha\gamma}{(\alpha-\gamma)(\alpha\gamma-1)},\quad\lambda_{2}=\lambda_{3}=-\cfrac{\alpha\gamma}{(\alpha-\gamma)(\alpha\gamma-1)}.
\]
Thus, from $\lvert\lambda_{1}\alpha^{-1}\rvert+\lvert\lambda_{2}\rvert+\lvert\lambda_{3}\rvert<\frac{1}{2}$
we get 
\[
\lvert(\alpha-\gamma)(\alpha\gamma-1)\rvert>2\alpha+2.
\]
Posing $\gamma=a+ib$, with some calculations we find 
\[
\alpha^{4}-4a\alpha^{3}+2(2a^{2}-7)\alpha^{2}-4(a+4)\alpha-3>0
\]
from which we have 
\[
\alpha>2+a+\sqrt{(a+2)^{2}+1}
\]
since $-1<a<1$ and $\alpha>1$. Using the explicit expression of
$\alpha$ and that $a=\frac{p-\sqrt{p^{2}-4q}}{4}$, we finally obtain
\[
\frac{1}{4}(p+\sqrt{-16+(-p-\sqrt{p^{2}-4q})^{2}}+\sqrt{p^{2}-4q})>2+\frac{p}{4}+\sqrt{1+\frac{1}{16}(8+p-\sqrt{p^{2}-4q})^{2}}-\frac{1}{4}\sqrt{p^{2}-4q},
\]
whose solutions are 
\[
\begin{cases}
2 \leq p\leq8,\quad-4-2p<q<\cfrac{p^{4}+8p^{3}-160p-400}{4p^{2}+32p+64}\\
p>8,\quad-4-2p<q<-4+2p
\end{cases}.
\]
\end{proof}

\end{document}